\pdfoutput=1 
\documentclass{amsart}
\usepackage[utf8]{inputenc}
\usepackage[T1]{fontenc}
\usepackage{graphicx}
\usepackage{grffile}
\usepackage{longtable}
\usepackage{wrapfig}
\usepackage{rotating}
\usepackage[normalem]{ulem}
\usepackage{amsmath}
\usepackage{textcomp}
\usepackage{amssymb}
\usepackage{capt-of}
\usepackage{hyperref}
\usepackage{etoolbox}
\usepackage{microtype}
\usepackage{mathrsfs}
\usepackage[english]{babel}
\usepackage{tikz}
\usepackage{tikz-cd}
\usepackage{amsthm}
\usepackage{thmtools}
\usepackage{thm-restate}
\usepackage{xpatch}
\declaretheorem[numberwithin=section]{theorem}  
\declaretheorem[sibling=theorem]{corollary}
\declaretheorem[sibling=theorem]{lemma}
\declaretheorem[sibling=theorem]{proposition}

\DeclareMathOperator\spec{Spec}

\DeclareMathOperator\pic{Pic}

\newcommand\dualab\hat
\newcommand\sh\mathscr

\makeatletter
\newcounter{proofstep}
\xpretocmd{\proof}{\setcounter{proofstep}{0}}{}{}
\newcommand{\proofstep}[1]{%
  \par
  \addvspace{\medskipamount}%
  \stepcounter{proofstep}%
  \noindent\emph{Step \theproofstep: #1}\par\nobreak\smallskip
  \@afterheading
}
\makeatother

\address{Department of Mathematics, Stony Brook University, Stony Brook, NY 11794-3651}
\email{mads.villadsen@stonybrook.edu}
\usepackage[backend=biber,style=alphabetic,giveninits=false,url=true,eprint=true,doi=false,isbn=false]{biblatex}
\renewbibmacro{in:}{%
\ifentrytype{article}{}{\printtext{\bibstring{in}\intitlepunct}}}
\author{Mads Bach Villadsen}
\date{}
\title{Kodaira dimension and zeros of holomorphic one-forms, revisited}
\hypersetup{
 pdfauthor={Mads Bach Villadsen},
 pdftitle={Kodaira dimension and zeros of holomorphic one-forms, revisited},
 pdfkeywords={},
 pdfsubject={},
 pdflang={English},
 final,
 bookmarks=true,
 bookmarksnumbered=true,
 bookmarksdepth=subsubsection}
\begin{document}

\maketitle
\begin{abstract}
We give a new proof that every holomorphic one-form on a smooth complex projective variety of general type must vanish at some point, first proven by Popa and Schnell using generic vanishing theorems for Hodge modules.
Our proof relies on Simpson's results on the relation between rank one Higgs bundles and local systems of one-dimensional complex vectors spaces, and the structure of the cohomology jump loci in their moduli spaces.
\end{abstract}

\section{Introduction}
\label{sec:org0094009}
In \cite{Popa2012}, Popa and Schnell showed that any holomorphic one-form on a smooth projective variety of general type must vanish at some point, a conjecture of Hacon-Kovács and Luo-Zhang \cite{Hacon2005,Luo2005}. Wei \cite{Wei2017} later gave a slightly simplified argument (as well as a generalization to log-one-forms). Both proofs use the decomposition theorem and various vanishing theorems for Hodge modules. We give a new approach using only classical Hodge theory, namely the rank one case of Simpson's correspondence between Higgs bundles and local systems, and his results on the structure of cohomology jump loci of local systems. Our approach should thus be much more accessible than either of the two previous proofs.

As in \cite{Popa2012}, we will prove the following more precise result.

\begin{theorem}[{\cite[Theorem 2.1]{Popa2012}}]
Let \(X\) be a smooth complex projective variety and \(f\colon X\to A\) a morphism to an abelian variety. If \(H^0(X,\omega_X^{\otimes d}\otimes f^{\ast}L^{-1})\ne 0\) for some integer \(d\ge 1\) and some ample line bundle \(L\) on \(A\), then for every holomorphic one-form \(\omega\) on \(A\), the pullback \(f^{\ast}\omega\) vanishes at some point of \(X\).
\label{thm:A}
\end{theorem}

The following conjecture of Luo and Zhang \cite{Luo2005} follows as in \cite{Popa2012}. For varieties of general type, this shows that every holomorphic one-form must vanish at some point.

\begin{corollary}[{\cite[Conjecture 1.2]{Popa2012}}]
Let \(X\) be a smooth complex projective variety and \(W\subseteq H^0(X,\Omega_X^1)\) be a linear subspace such that every element of \(W\setminus \{0\}\) is everywhere non-vanishing. Then \(\dim W\le\dim X-\kappa(X)\).
\end{corollary}

Our proof of Theorem \ref{thm:A} goes as follows. Let \(V=H^0(A,\Omega_A^1)\) and
\begin{align*}
Z_f=\{(x,\omega)\in X\times V\mid f^{\ast}\omega(T_xX)=0\}.
\end{align*}
The goal is to show that the restriction of the projection \(p_2\colon X\times V\to V\) to \(Z_f\) is surjective.

We borrow the idea in \cite{Popa2012}, going back to work of Viehweg-Zuo \cite{Viehweg2000}, of constructing two separate sheaves on \(V\). The first is an ambient sheaf coming from a cyclic cover of \(X\), which we will show to be locally free. The second is a non-zero subsheaf coming from \(X\) and supported on \(p_2(Z_f)\). As the subsheaf is necessarily torsion free, it must have support equal to \(V\), hence \(p_2(Z_f)=V\) as desired.

Let us outline the argument for why the ambient sheaf is locally free.

After base change by an isogeny of \(A\), we can assume that \((\omega_X\otimes f^{\ast}L^{-1})^{\otimes d}\) has a non-zero section \(s\). Let \(Y\) be a resolution of singularities of the associated degree \(d\) cyclic cover of \(X\) branched along \(s\), and consider the composition \(h\colon Y\to A\).

The ambient sheaf is a higher direct image of a complex of sheaves on \(Y\times V\), and the fibres of the complex over points in \(V\) are Dolbeault complexes of certain Higgs bundles on \(Y\). Using Simpson's results \cite{Simpson1992,Simpson1993} relating Higgs bundles to local systems and analyzing cohomology jump loci in the moduli space of local systems, we show that the hypercohomology groups of these Dolbeault complexes have constant dimension over \(V\). This gives the result by Grauert's theorem on locally free direct images.

\subsection*{Acknowledgements}
I would like to thank my advisor Christian Schnell for pointing me to this topic, and for many enlightening remarks. I also thank Ben Wu for related discussions and useful comments on drafts of the paper.

\section{The proof}
\label{sec:orgda44d7e}
Fix a smooth projective variety \(X\) over the complex numbers and a morphism \(f\colon X\to A\) to an abelian variety throughout. Let \(V=H^0(A,\Omega_A^1)\) be the vector space of holomorphic one-forms on \(A\), and let \(S=\operatorname{Sym}V^{\ast}\) be the graded coordinate ring of the vector space \(V\). For an integer \(i\), let \(S_{\bullet+i}\) denote \(S\) as a graded module over itself, with grading shifted by \(i\), and let \(C_{X,\bullet}\) be the complex of graded \(\mathcal{O}_X\otimes S\)-modules given by
\begin{align*}
\mathcal{O}_X\otimes S_{\bullet-g}\to\Omega_X^1\otimes S_{\bullet-g+1}\to\cdots\to\Omega_X^n\otimes S_{\bullet-g+n},
\end{align*}
in degrees \(-n\) to \(0\), where \(n=\dim X,g=\dim A\), and the differential is induced by the map \(\mathcal{O}_X\otimes V\to\Omega_X^1\) given by \(\phi\otimes\omega\mapsto \phi f^{\ast}\omega\). In a basis \(\omega_1,\ldots,\omega_g\) of \(V\) with dual basis \(s_1,\ldots,s_g\) of \(S_1\), the differential is given by
\begin{align*}
\theta\otimes s\mapsto\sum_{i=1}^g (\theta\wedge f^{\ast}\omega_i)\otimes s_is.
\end{align*}
We denote the associated complex of vector bundles on \(X\times V\) by \(C_X\).

\begin{lemma}[{\cite[Lemma 14.1]{Popa2012}}]
The support of \(C_X\) is equal to
\begin{align*}
Z_f=\{(x,\omega)\in X\times V\mid f^{\ast}\omega(T_xX)=0\}.
\end{align*}
\label{lemma:support}
\end{lemma}

For \(\alpha\in\pic^0(A)\), let \(C_X^{\alpha}=C_X\otimes p_1^{\ast}f^{\ast}\alpha\) and \(C_{X,\bullet}^{\alpha}=C_{X,\bullet}\otimes f^{\ast}\alpha\), where \(p_1\) is the projection \(X\times V\to X\). The sheaves \(R^ip_{2\ast}C_X^{\alpha}\) on \(V\) are then supported on \(p_2(Z_f)\) for all \(i\); recall that we are trying to show that \(p_2(Z_f)=V\). We will show that these sheaves are locally free for general \(\alpha\) in Proposition \ref{prop:generically-locally-free-on-V} below. As we will see, the fibres of \(C_X^{\alpha}\) over \(V\) are related to certain Higgs bundles on \(X\).

Recall that a Higgs bundle on \(X\) is a vector bundle \(E\) together with a morphism of coherent sheaves \(\theta\colon E\to\Omega_X^1\otimes E\), the Higgs field, satisfying \(\theta\wedge\theta=0\). Given a Higgs bundle, we get a holomorphic Dolbeault complex
\begin{align*}
E\xrightarrow{\theta\wedge} E\otimes\Omega_X^1\to\cdots\to E\otimes\Omega_X^n.
\end{align*}

Simpson's non-abelian Hodge theorem \cite{Simpson1992} associates to each Higgs bundle \((E,\theta)\) (satisfying some conditions on stability and Chern classes) a local system \(\mathbb{C}_{(E,\theta)}\) of complex vector spaces, and shows that Dolbeault cohomology
\begin{align*}
H^k_{\mathrm{Dol}}(X,E,\theta)=\mathbf{H}^k(X,E\xrightarrow{\theta\wedge} E\otimes\Omega_X^1\to\cdots\to E\otimes\Omega_X^n),
\end{align*}
the hypercohomology of the Dolbeault complex, is isomorphic to the cohomology of \(\mathbb{C}_{(E,\theta)}\).

We will only need the rank one case; see also the lecture notes \cite[Lectures 17-18]{Schnell2013} for a concrete treatment of this case, and the associated Hodge theory.

In the rank one case, a Higgs bundle is just a line bundle together with a holomorphic one-form. The stability condition in Simpson's theorem is always satisfied for line bundles, and the condition on Chern classes is simply that the first Chern class vanishes in \(H^2(X,\mathbb{C})\); let \(\pic^{\tau}(X)\) be the space of line bundles satisfying this condition. Let then \(M_{\mathrm{Dol}}(X)=\pic^{\tau}(X)\times H^0(X,\Omega_X^1)\), and let \(M_{\mathrm{B}}(X)\) denote the moduli space of local systems of one-dimensional complex vector spaces on \(X\). Then Simpson's correspondence, mapping a rank one Higgs bundle to the associated local system, takes the form of a real analytic isomorphism \(M_{\mathrm{Dol}}(X)\cong M_{\mathrm{B}}(X)\).

For each \(k\) and \(m\), consider the cohomology jump loci
\begin{align*}
\Sigma_{m}^k(X)&=\{\mathcal{L}\in M_{\mathrm{B}}(X)\mid \dim H^k(X,\mathcal{L})\ge m\}\\
\Sigma_{m}^k(X)_{\mathrm{Dol}}&=\{(E,\theta)\in M_{\mathrm{Dol}}(X)\mid \dim H^k_{\mathrm{Dol}}(X,E,\theta)\ge m\}
\end{align*}
of local systems and Dolbeault cohomology of Higgs bundles. These loci get mapped to each other under Simpson's correspondence.

Using this relationship, Simpson \cite{Simpson1993} proves that every irreducible component of these loci is a linear subvariety or, in his terminology, a translate of a triple torus (in fact a torsion translate, though we will not need that). A triple torus is a closed, connected, algebraic subgroup \(N\) of \(M_{\mathrm{B}}(X)\) such that the corresponding subgroup of \(M_{\mathrm{Dol}}(X)\) (which we will also refer to as a linear subvariety) is also algebraic (this is equivalent to the usual definition, involving also the de Rham moduli space, by \cite[Lemma 2.1]{Simpson1993}). A linear subvariety is thus a subset of \(M_{\mathrm{B}}(X)\) of the form
\begin{align*}
\{\mathcal{L}\otimes \mathcal{N}\mid \mathcal{N}\in N\}
\end{align*}
where \(N\) is a triple torus and \(\mathcal{L}\in M_{\mathrm{B}}(X)\) a local system. 

Simpson \cite[Lemma 2.1]{Simpson1993} shows that a triple torus is of the form \(g^{\ast}M_{\mathrm{B}}(T)\) for a map \(g\colon X\to T\) to an abelian variety, where \(g^{\ast}\colon M_{\mathrm{B}}(T)\to M_{\mathrm{B}}(X)\) denotes pullback of local systems. It follows that a linear subvariety in \(M_{\mathrm{Dol}}(X)\) is a translate of a subset of the form \(g^{\ast}\pic^0(T)\times g^{\ast}H^0(T,\Omega_T^1)\) for \(g\colon X\to T\) a morphism to an abelian variety. In particular, a linear subvariety is either the entire moduli space, or maps to a proper subvariety of \(\pic^0(X)\) under the projection \(M_{\mathrm{Dol}}(X)\to\pic^0(X)\) that forgets the Higgs field.

The following proposition is the main new ingredient in the proof. Note that this proposition is valid for an arbitrary morphism \(f\colon X\to A\), not just those that satisfy the hypotheses of Theorem \ref{thm:A}.

\begin{proposition}
For general \(\alpha\in\pic^0(A)\), the higher direct image sheaves \(R^ip_{2\ast}C_X^{\alpha}\) are locally free on \(V\) for all \(i\).
\label{prop:generically-locally-free-on-V}
\end{proposition}
\begin{proof}
We will show that for general \(\alpha\), the dimensions of the hypercohomology \(H^i(X\times \{v\},C_X^{\alpha}|_{X\times \{v\}})\) of \(C_X^{\alpha}\) on fibres of \(p_2\) are constant in \(v\). The result follows by a version of Grauert's theorem on locally free direct images for complexes of sheaves \cite[Proposition 7.8.4]{EGAIII}

Note that for any fibre \(X\times \{v\}\) of \(p_2\) for \(v\in V\), the restriction of \(C_X^{\alpha}\) to the fibre is the Dolbeault complex 
\begin{align*}
f^{\ast}\alpha\xrightarrow{\wedge f^{\ast}v} f^{\ast}\alpha\otimes\Omega_X^1\to\cdots\to f^{\ast}\alpha\otimes\Omega_X^n.
\end{align*}
of the Higgs bundle \((f^{\ast}\alpha,f^{\ast}v)\). The Dolbeault cohomology of these Higgs bundles is governed by the cohomology jump loci \(\Sigma_{m}^k(X)_{\mathrm{Dol}}\), of which only finitely many are nonempty by algebraicity. Each irreducible component of the nonempty ones is a linear subvariety, so it suffices to show that for each linear subvariety \(S\) of \(M_{\mathrm{Dol}}(X)\), the set \(\{\alpha\}\times f^{\ast}V\) is either entirely contained in \(S\) or entirely disjoint from it, for general \(\alpha\).

Let then \(\phi=f^{\ast}\colon M_{\mathrm{Dol}}(A)\to M_{\mathrm{Dol}}(X)\), and suppose \(S\subset M_{\mathrm{Dol}}(X)\) is a linear subvariety. We observe that if \(N\) is a triple torus, then the connected component of \(\phi^{-1}(N)\) is again a triple torus; it follows that \(\phi^{-1}(S)\) is either empty, or a finite union of linear subvarieties. If \(\phi^{-1}(S)=M_{\mathrm{Dol}}(A)\) then \(\{\alpha\}\times f^{\ast}V\subset S\) for any \(\alpha\in\pic^0(A)\). If \(S\) is a proper subset of \(M_{\mathrm{Dol}}(A)\), it suffices to take \(\alpha\) to be outside the image of \(\phi^{-1}(S)\) in \(\pic^0(A)\) under the projection \(M_{\mathrm{Dol}}(A)\to\pic^0(A)\).
\end{proof}

If we could show that one of the sheaves \(R^ip_{2\ast}C_X^{\alpha}\) were nontrivial on \(V\), under the hypotheses of Theorem \ref{thm:A}, we would now be done. Unfortunately we cannot, but instead we make use of a covering construction as in \cite{Popa2012}.

\begin{lemma}[{\cite[Lemma 11.1]{Popa2012}}]
Suppose \(\omega_X^{\otimes d}\otimes f^{\ast}L^{-1}\) has a nonzero section for some \(d\) and some ample line bundle \(L\) on \(A\). For an isogeny \(\phi\colon A'\to A\), define \(f'\colon X'\to A'\) by base change of \(f\). For an appropriately chosen \(\phi\), there exists an ample line bundle \(L'\) on \(A'\) such that \((\omega_{X'}\otimes f'^{\ast}L'^{-1})^{\otimes d}\) has a nonzero section.
\label{lemma:isogeny}
\end{lemma}

Assume now the hypotheses of Theorem \ref{thm:A}. Note that zero loci of one-forms are not affected by étale covers, so if we can prove the theorem for \(f'\colon X'\to A'\), then the desired conclusion also follows for \(f\colon X\to A\). 

In particular, replacing \(f\) by this \(f'\), we can now assume without loss of generality that \(B^{\otimes d}\) has a nonzero section \(s\) for \(B=\omega_X\otimes f^{\ast}L^{-1}\). Let \(Y\) be a resolution of singularities of the \(d\)-fold cyclic cover \(\pi\colon X_d\to X\) ramified along \(Z(s)\), giving us the following maps:
\begin{equation*}
\begin{tikzcd}
Y\ar[rr,bend left,"\phi"]\ar[r]\ar[rrd,"h"]&X_d\ar[r,"\pi"]&X\ar[d,"f"]\\
&&A
\end{tikzcd}
\end{equation*}

By construction, \(X_d=\spec \bigoplus_{i=0}^{d-1}B^{-i}\), so \(\pi_{\ast}\pi^{\ast}B=\bigoplus_{i=-1}^{d-2}B^{-i}\). This has a section in the \(i=0\) term, and the corresponding section of \(\pi^{\ast}B\) gives a morphism \(\phi^{\ast}B^{-1}\to \mathcal{O}_Y\), an isomorphism away from \(Z(s)\). Together with pullback of forms, this gives injective morphisms \(\phi^{\ast}(B^{-1}\otimes\Omega_X^k)\to\Omega_Y^k\). As \(\mathcal{O}_X\to\phi_{\ast}\mathcal{O}_Y\) is injective, the corresponding morphisms
\begin{align*}
B^{-1}\otimes\Omega_X^k\to\phi_{\ast}\Omega_Y^k
\end{align*}
on \(X\) are also injective.

Note that we get a complex \(C_{Y,\bullet}\) of graded \(\mathcal{O}_Y\otimes S\)-modules using the morphism \(h\colon Y\to A\), constructed in the same way that \(C_{X,\bullet}\) was constructed starting from \(f\) above Lemma \ref{lemma:support}.

We give a slightly modified version of \cite[Lemma 13.1]{Popa2012}.

\begin{lemma}
The morphisms above induce a morphism of complexes of graded \(\mathcal{O}_X\otimes S\)-modules
\begin{align*}
B^{-1}\otimes C_{X,\bullet}\to \mathbf{R}\phi_{\ast}C_{Y,\bullet}
\end{align*}
\label{lemma:morphism-to-C_Y}
\end{lemma}
\begin{proof}
The morphisms \(\phi^{\ast}(B^{-1}\otimes \Omega_X^k)\to\Omega_Y^k\)
commute with the differentials on \(Y\), giving
\begin{align*}
\phi^{\ast}(B^{-1}\otimes C_{X,\bullet})\to C_{Y,\bullet}
\end{align*}
Using the projection formula and the morphism \(\mathcal{O}_X\to \mathbf{R}\phi_{\ast}\mathcal{O}_Y\), pushing forward to \(X\) gives the desired composition
\begin{align*}
B^{-1}\otimes C_{X,\bullet}\to (B^{-1}\otimes C_{X,\bullet})\otimes^{\mathbf{L}}\mathbf{R}\phi_{\ast}\mathcal{O}_Y\to \mathbf{R}\phi_{\ast}C_{Y,\bullet}
\end{align*}
\end{proof}

\begin{proof}[{Proof of Theorem \ref{thm:A}}]
We must show that \(Z_f\) surjects onto \(V\) under the projection \(p_2\colon X\times V\to V\).

Let \(\alpha\in\pic^0(A)\) be a general element. Then Lemma \ref{lemma:morphism-to-C_Y} gives, after twisting by \(f^{\ast}\alpha\) and pushing forward to \(V\), a morphism \(\mathbf{R}p_{2\ast}(p_1^{\ast}B^{-1}\otimes C_X^{\alpha})\to \mathbf{R}p_2^{\ast}C_Y^{\alpha}\) where \(p_1\colon X\times V\to X\) is the first projection, and \(p_2\), by abuse of notation, is used for both of the projections \(X\times V\to V\) and \(Y\times V\to V\). Let \(\sh F\) be the image of the induced map \(R^0p_{2\ast}(p_1^{\ast}B^{-1}\otimes C_X^{\alpha})\to R^0p_{2\ast}C_Y^{\alpha}\).

As \(\alpha\) is general, each \(R^ip_{2\ast}C_Y^{\alpha}\) is locally free by Proposition \ref{prop:generically-locally-free-on-V}. In particular \(\sh F\) is torsion free. Since \(C_X\) is supported on \(Z_f\), \(\sh F\) is supported on \(p_2(Z_f)\), so it suffices to show that \(\sh F\) is non-zero.

Let \(k=g-n\).  Then \(C_{X,k}=\omega_X\) and \(C_{Y,k}=\omega_Y\), and the morphism \(B^{-1}\otimes C_{X,k}\to \mathbf{R}\phi_{\ast}C_{Y,k}\) from Lemma \ref{lemma:morphism-to-C_Y} is just the morphism of sheaves \(f^{\ast}L=B^{-1}\otimes\omega_X\to\phi_{\ast}\omega_Y\) constructed before the lemma. Indeed \(\mathbf{R}\phi_{\ast}\omega_Y=\phi_{\ast}\omega_Y\) since \(\phi\) is generically finite, by results of Kollár \cite{Kollar1986a}.

After twisting by \(\alpha\), the morphism \(B^{-1}\otimes C_{X,k}^{\alpha}\to \mathbf{R}\phi_{\ast}C_{Y,k}^{\alpha}\) thus gets identified with \(f^{\ast}(L\otimes\alpha)\to\phi_{\ast}\omega_Y\otimes f^{\ast}\alpha\). For the graded \(S\)-module \(\sh F_{\bullet}=H^0(V,\sh F)\), it follows that \(\sh F_k\cong H^0(X,f^{\ast}(L\otimes\alpha))\) since the pushforward to \(V\) preserves injectivity.
But \(f^{\ast}(L\otimes\alpha)\) has non-zero sections; otherwise all sections of its pushforward \(f_{\ast}\mathcal{O}_X\otimes L\otimes\alpha\) to \(A\) would vanish, which would imply that \(X\) is contained in a general translate of a hyperplane section of \(A\), a contradiction. Thus \(\sh F\) is non-zero.
\end{proof}

\defbibheading{bibliographysection}[\refname]{
  \section*{#1}
}
\printbibliography[heading=bibliographysection]
\end{document}